\documentclass[12pt]{amsart}
\usepackage{geometry}
\usepackage{mathtools}

\usepackage{amssymb,amsmath,amsthm}
\usepackage[english]{babel}
\usepackage{hyperref}

\newcommand{\length}{\operatorname{\lambda}}

\newcommand{\mf}[1]{\mathfrak #1}

\DeclareMathOperator{\eh}{e}
\DeclareMathOperator{\ehk}{e_{HK}}

\DeclareMathOperator{\ord}{ord}

\renewcommand{\frq}[1]{{#1}^{[q]}}

\newtheorem{theorem}{Theorem}
\newtheorem{lemma}[theorem]{Lemma}
\newtheorem{proposition}[theorem]{Proposition}
\newtheorem{corollary}[theorem]{Corollary}
\theoremstyle{definition}
\newtheorem{definition}[theorem]{Definition}

\theoremstyle{remark}

\newtheorem{question}[theorem]{Question}
\newtheorem{conjecture}[theorem]{Conjecture}

\numberwithin{theorem}{section}
\numberwithin{equation}{section}

\begin{document}

\title[Hilbert--Kunz multiplicity of the powers of an ideal]
{Hilbert--Kunz multiplicity of the powers of an ideal}

\author{Ilya Smirnov}
\address{Department of Mathematics, Stockholm University, S-106 91, Stockholm, Sweden}
\email{smirnov@math.su.se}

\date{\today}

\begin{abstract}
We study Hilbert--Kunz multiplicity of the powers of an ideal and establish existence of the second coefficient at
the full level of generality, thus extending a recent result of Trivedi.
We describe the second coefficient as the limit of the Hilbert coefficients of Frobenius powers 
and show that it is additive in short exact sequences and satisfies a Northcott-type inequality. 

\end{abstract}

\maketitle

\section{Introduction}

This note studies Hilbert--Kunz multiplicity -- a multiplicity theory native to positive characteristic which
mimics the definition of Hilbert--Samuel multiplicity but replaces regular powers $I^n$ by 
Frobenius powers $I^{[p^n]} := \{x^{p^n} \mid x \in I\}$.
\begin{definition}
Let $(R, \mf m)$ be a local ring of positive characteristic $p > 0$ and dimension $d$ 
and $I$ be an $\mf m$-primary ideal.
Then the Hilbert--Kunz multiplicity of a finite $R$-module $M$ with respect to $I$ is 
\[
\ehk(I, M) = \lim_{q \to \infty} \frac{\length (M/\frq{I}M)}{q^{d}},
\]
where $q = p^e$ is a varying power of $p$.
\end{definition}
As its name suggests, Hilbert--Kunz multiplicity originates from the work of Kunz (\cite{Kunz1, Kunz2}),
who initiated the study of the sequence giving its definition. The existence of the limit was proven
later by Monsky (\cite{Monsky}). 

It is natural to seek relations between the classical Hilbert--Samuel theory of multiplicity and the new theory. 
One such relation, in particular, is given by inequalities 
\[
\frac 1{d!} \eh(I) \leq \ehk (I) \leq \eh(I),
\]
where the left inequality directly follows from the inclusion $\frq{I} \subseteq I^q$. 
It was shown by Hanes (\cite{Hanes2}) that the left inequality is never sharp, 
but Watanabe and Yoshida (\cite[Theorem~1.1]{WatanabeYoshidaTwo}) and Hanes (\cite[Corollary~II.7]{Hanes})
proved that the inequality is sharp asymptotically, {\it i.e.}, 
$\lim\limits_{k \to \infty} \frac{d!\ehk (I^k)}{\eh(I^k)} = 1$.
Because $\eh(I^k) = k^d \eh(I)$, this result can be restated as 
$\lim\limits_{k \to \infty} \frac{d!\ehk(I^k)}{k^d} = \eh(I)$.
 Later, Hanes (\cite[Theorem~3.2]{Hanes2}) improved this  by showing that
\[
\length (R/(\frq{I})^k) = \left( \frac{\eh(I)}{d!}k^d + O(k^{d-1})\right) q^d.
\]

Recently, Trivedi gave a further insight to the problem by describing the $O(k^{d-1})$ term in the formula. 
Recall that the Hilbert coefficients of $I$, $\eh_i(I)$, are defined by  the Hilbert--Samuel polynomial: 
for all $k \gg 0$
\[
\length (R/I^k) = \sum_{i = 0}^{d} (-1)^i \binom{k + d - 1 - i}{d - i} \eh_i (I),
\]
where $\eh_0(I) =\eh(I)$.
In \cite{Trivedi} Trivedi showed that if $R$ is a standard graded ring and $I$ 
is a homogeneous ideal generated in same degree, then 
\[
\lim_{k \to \infty} \frac{\ehk(I^k) - \eh(I^k)/d!}{k^{d - 1}} = 
\frac{\eh(I)}{2(d-2)!} - \lim_{q \to \infty} \frac{\eh_1 (\frq{I})}{(d-1)!q^d},
\]
and the last limit exists.
The goal of this article is to provide a proof of this result in a full level of generality, 
{\it i.e.}, when $(R, \mf m)$ is local and $I$ is an arbitrary $\mf m$-primary ideal. 

We also want to point out that this result can be restated in the following form:
\[
\lim_{k \to \infty} \left( (d-1)! \frac{\ehk(I^k) - \binom{k + d - 1}{d} \eh (I) }{k^{d-1}} \right)
= - \lim_{q \to \infty} \frac{e_1(\frq{I})}{q^d}.
\]
Comparing this with the formula defining the Hilbert coefficients, we want to pose the following questions. 

\begin{question}
Does the limit 
\[
\lim_{q \to \infty} \frac{\eh_i (\frq{I})}{q^d}
\]
exist for all $i$?
\end{question}

\begin{question}\label{q decomposition}
Do we have for $k \gg 0$ that 
\[
\ehk(I^k) = \sum_{i = 0}^{d} (-1)^i \binom{k + d - 1 - i}{d - i} \lim_{q \to \infty} \frac{\eh_i (\frq{I})}{q^d}?
\]
\end{question}

It is not surprising that $\ehk(I^k)$ should be eventually a polynomial. 
For example, if $(R, \mf m)$ is a finite subring of a regular local ring $(S, \mf n)$, then 
$[S: R]\ehk(I) = \ehk (IS)[S/\mf n : R/\mf n]$ (\cite{WatanabeYoshida}) and, 
since Frobenius is flat in $S$ (\cite{Kunz1}), $\ehk(IS) = \length (S/IS)$.
Thus $\ehk(I^k)$ is a multiple of the Hilbert--Samuel polynomial of $IS$ in $S$.  
The hard task is to show that the Hilbert coefficients of $\frq{I}$ have the prescribed limit.
We believe that this is explained by an analogue of Proposition~\ref{addi} for further coefficients. 

\subsection{Methods}

Trivedi's proof uses uniform convergence of 
Hilbert--Kunz density function that she developed in \cite{Trivedi2},
thus her methods can be applied only in the graded setting. 
Our approach uses uniform convergence techniques for Hilbert--Kunz function
pioneered by Tucker in \cite{Tucker} and allows us to simplify the proofs and generalize the results of 
\cite{Trivedi}. 
Using the improved techniques, we 
are able to show the additivity property of the new limit invariant defined by 
$\eh_1 (\frq{I})$ (Proposition~\ref{addi}).

\section{Main results}\label{main section}

\subsection{A uniform convergence result}\label{HK equi convergence}

First, we establish a refinement of \cite[Lemma~3.2]{Tucker}.

\begin{lemma}\label{boundlemma}
Let $(R, \mf m)$ be a local ring of characteristic $p > 0$ and 
$M$ be a finitely generated $R$-module.
Then for every $\mf m$-primary ideal $I$ 
there exists a constant $C$ such that for all $q, k \geq 1$ we have
\[
\length (M/(\frq{I})^kM) < C (qk)^{\dim M}.
\]
\end{lemma}
\begin{proof}
Let $\mu $ be the number of generators of $I$, then $I^{\mu q} \subseteq \frq{I}$, 
so 
\[
\length (M/(\frq{I})^kM) \leq \length (M/I^{\mu kq}M).
\]
Using the Hilbert--Samuel function of $M$ with respect to $I$ we may find a constant $B$ such that
$\length (M/I^nM) \leq B n^{\dim M}$, so for all $k$ and $q$ we have
\[
\length (M/I^{\mu kq}M) \leq B(\mu k q)^{\dim M} = B (\mu)^{\dim M} (kq)^{\dim M}.
\]
Thus the claim follows for the constant $C : = B (\mu)^{\dim M}$. 
\end{proof}

Now following \cite[Lemma~3.3]{Tucker} we can easily obtain the following lemma.

\begin{lemma}\label{min primes}
Let $(R, \mf m)$ be a reduced local ring  of dimension $d > 0$ and characteristic $p > 0$ and $M,N$ be finitely generated $R$-modules.
Suppose $M_\mf p \cong  N_\mf p$ for any minimal prime $\mf p$ such that $\dim R/\mf p= \dim R$.
Then for every $\mf m$-primary ideal $I$ 
there exists a constant $C$ such that for all $k,q \geq 1$ 
we have
\[
|\length (M/(\frq{I})^kM) - \length (N/(\frq{I})^kN)| < C (kq)^{d - 1}.
\] 
\end{lemma}
%

The lemma now can be applied to modules $F_* M$ and $M$. 
Recall that we use $F_* M$ to denote an $R$-module obtained from $M$ 
via the restriction of scalars through the Frobenius endomorphism $F \colon R \to R$.
Thus $F_* M$ is isomorphic to $M$ as an abelian group, but elements of $R$ act as $p$-powers.
So, for any ideal $I$, $IF_* M \cong F_* I^{[p]}M$. 

The ranks of $M$ and $F_*^e M$ can be compared via a result of
Kunz (\cite[Proposition~2.3]{Kunz2}) who observed how localization affects the residue field:
if $\mf p \subseteq \mf q$ are prime ideals then
\[[k(\mf p) : k(\mf p)^p] = [k(\mf q) : k(\mf q)^p] p^{\dim R_\mf q/\mf pR_\mf q}.\]
It follows that in a reduced local ring $(R, \mf m, k)$, for every minimal prime $\mf p$ such that 
$\dim R/\mf p = \dim R$, the vector spaces $(F_* M)_\mf p \cong F_* M_\mf p$
and $\oplus^{p^d [k : k^p]} M_\mf p$ have equal dimension.  
Now, as in \cite[Section~3]{equi} we may use the improved constant of Lemma~\ref{boundlemma}
to follow \cite[Proposition~3.4, Corollary~{3.5}, Theorem~3.6]{Tucker} 
and obtain the following result.

\begin{theorem}\label{convergence}
Let $(R, \mf m)$ be a local ring of dimension $d > 0$ and characteristic $p > 0$, 
$I$ an $\mf m$-primary ideal, 
and $M$ be a finitely generated $R$-module.
There exists a constant $C$ and a constant $q_0 \geq 1$
such that for every $q, q',k \geq 1$ 
\[
\left |\frac{\length (M/(I^{[q_0q]})^kM)}{q^d} - \frac{\length (M/(I^{[q_0q'q]})^kM)}{(q')^d q^d} \right| 
< C \frac{k^{d-1}}{q}.
\]
\end{theorem}

\begin{corollary}\label{uni}
Let $(R, \mf m)$ be a local ring of dimension $d > 0$ and characteristic $p > 0$, $I$ be an $\mf m$-primary ideal,
and $M$ be a finitely generated $R$-module.
Then the bisequence
\[
\frac{\length (M/(\frq{I})^kM)}{q^d k^{d - 1}}
\]
converges uniformly, independently of $k$, to its limit $\ehk (I^k, M)/k^{d-1}$.
\end{corollary}
\begin{proof}
First, observe that 
\[
\lim_{q' \to \infty} \frac{\length (M/I^{[q_0q'q]}M)}{(q')^d q^d}
= \ehk (I^{[q_0]}, M) = q_0^d \ehk (I, M).
\]
Thus, after letting let $q' \to \infty$ in Theorem~\ref{convergence} it follows that
\[
\left |\frac{\length (M/(I^{[qq_0]})^kM)}{q^d} -  q_0^d \ehk (I^k, M)\right| 
< C \frac{k^{d-1}}{q}.
\]
Hence, the assertion is obtained by replacing $C$ by $C/q_0^d$ and $q$ by $qq_0$. 
\end{proof}

\subsection{Existence of the limit and its addivity}

We want to use the uniform convergence via the following standard result:
if for a bisequence $a_{q, k}$ 
\begin{itemize}
\item $\lim\limits_{q \to \infty } a_{q, k}$ exists uniformly of $k$, and
\item $\lim\limits_{k \to \infty} a_{q, k}$ exists for all (sufficiently large) $q$,
\end{itemize}
then $\lim\limits_{q,k \to \infty} a_{q,k}$ exists and the iterated limits exist,
and they are all equal.

\begin{proposition}\label{decompose}
Let $(R, \mf m)$ be a local ring of dimension $d > 0$ and characteristic $p > 0$, $M$ be a finitely generated $R$-module, 
 and $I$ be an $\mf m$-primary ideal. Then the following two limits exist and are equal
\[
\lim_{k \to \infty} \left( (d-1)! \frac{\ehk(I^k, M) - \binom{k + d - 1}{d} \eh (I, M) }{k^{d-1}} \right)
= - \lim_{q \to \infty} \frac{e_1(\frq{I}, M)}{q^d}.
\]

In other words,  
\[
\ehk(I^k, M) =  \eh (I, M) \binom{k + d - 1}{d} - \lim_{q \to \infty} \frac{e_1(\frq{I}, M)}{q^d} \binom{k + d - 2}{d - 1} + o(k^{d - 1}).
\]
\end{proposition}
\begin{proof}
Observe that in the left side of the assertion we have 
\[
\lim_{k \to \infty} \frac{\ehk(I^k, M) - \binom{k + d - 1}{d} \eh (I, M) }{k^{d-1}} 
= 
\lim_{k \to \infty} \lim_{q \to \infty}
\frac{\length (M/\frq{(I^k)}M) - \binom{k + d - 1}{d} \eh (\frq{I}, M) }{q^dk^{d-1}}.
\]
Now, we may use Corollary~\ref{uni} to interchange the order of limits and get that
\[
\lim_{q \to \infty} \lim_{k \to \infty}
\left( (d-1)! \frac{\length (M/(\frq{I})^kM) - \binom{k + d - 1}{d} \eh (\frq{I}, M) }{q^dk^{d-1}} \right)
= - \lim_{q \to \infty} \frac{\eh_1 (\frq{I}, M)}{q^d}.
\]

\end{proof}

We record the following additive property. 

\begin{proposition}\label{addi}
Let $(R, \mf m)$ be a local ring of dimension $d > 0$ and characteristic $p > 0$ and $I$ be an $\mf m$-primary ideal. 
Let 
\[
0 \to L \to M \to N \to 0
\]
be an exact sequence of finitely generated $R$-modules. 
Then 
\[
\lim_{q \to \infty} \frac{\eh_1(\frq{I}, M)}{q^d} = 
\lim_{q \to \infty} \frac{\eh_1(\frq{I}, N)}{q^d} + \lim_{q \to \infty} \frac{\eh_1(\frq{I}, L)}{q^d}.
\]
\end{proposition}
\begin{proof}
First of all, Hilbert--Kunz multiplicity is additive in short exact sequences, so 
by Corollary~\ref{uni} for any $\varepsilon > 0$ there exists $q_0$ such that for all $k \geq 1$ and all $q \geq q_0$
we have
\[
\frac{1}{q^dk^{d- 1}} 
\left| 
\length (M/(\frq{I})^k M) - \length (L/(\frq{I})^k L) - \length (N/(\frq{I})^k N)
\right| < \frac{\varepsilon}{4}.
\]
For any given $q \geq q_0$ and $\varepsilon$, we may find $k$ such that 
\[
\left | \frac{1}{k^{d-1}} \left (\length (M/(\frq{I})^k M) -\eh (\frq{I}, M)\binom{k + d - 1}{d} \right) + \eh_1(\frq{I}, M) \right| < \frac{\varepsilon}{4}
\]
and similarly for $N$ and $L$.

Furthermore,  we may rewrite
\begin{align*}
&\left |\eh_1(\frq{I}, M) - \eh_1(\frq{I}, N) - \eh_1(\frq{I}), L) \right | \leq \\ 
&\left | \frac{1}{k^{d-1}} \left (\length (M/(\frq{I})^k M) - \length (N/(\frq{I})^k N) - \length (L/(\frq{I})^k L) \right )
+ \eh_1(\frq{I}, M) - \eh_1(\frq{I}, N) - \eh_1(\frq{I}, L) \right | 
\\&+ \frac{1}{k^{d-1}}\left |\length (M/(\frq{I})^k M) - \length (N/(\frq{I})^k N) - \length (L/(\frq{I})^k L) \right |.
\end{align*}
But multiplicity is additive in short exact sequences, so 
\begin{align*}
&\left | \frac{1}{k^{d-1}} (\length (M/(\frq{I})^k M) - \length (N/(\frq{I})^k N) - \length (L/(\frq{I})^k L))
+ \eh_1(\frq{I}, M) - \eh_1(\frq{I}, N) - \eh_1(\frq{I}, L) \right | \leq \\
&\left | \frac{1}{k^{d-1}} \left (\length (M/(\frq{I})^k M) -\eh (\frq{I}, M)\binom{k + d - 1}{d} \right) + \eh_1(\frq{I}, M) \right| + \\
&\left |\frac{1}{k^{d-1}} \left (\length (N/(\frq{I})^k N) - \eh(\frq{I}, N)\binom{k + d - 1}{d} \right) + \eh_1(\frq{I}, N) \right | + \\
&\left |\frac{1}{k^{d-1}} \left (\length (L/(\frq{I})^k L) - \eh(\frq{I}, L)\binom{k + d - 1}{d} \right) + \eh_1(\frq{I}, L) \right | < \frac 3 4 \varepsilon. 
\end{align*}
Thus, combining the estimates, we get that for any $q \geq q_0$, 
\[
\frac{1}{q^d} \left |\eh_1(\frq{I}, M) - \eh_1(\frq{I}, N) - \eh_1(\frq{I}), L) \right | < 
\frac 1 {q^d} \frac 34 \varepsilon + \frac 1 4 \varepsilon < \varepsilon.
\]
\end{proof}

Thus, we generalize \cite[Proposition~3.8]{Trivedi}.

\begin{corollary}
Let $(R, \mf m)$ be a local ring of dimension $d > 0$ and characteristic $p > 0$.
Then for any finitely generated $R$-module $M$ we have
\[
\lim_{q \to \infty} \frac{\eh_1(\frq{I}, M)}{q^d} = 
\sum_{\dim R/\mf p = d} 
\length (M_\mf p) \lim_{q \to \infty} \frac{\eh_1(\frq{I}, R/\mf p)}{q^d}.
\]
\end{corollary}
\begin{proof}
Apply Proposition~\ref{addi} to a prime filtration of $M$. 
\end{proof}

\subsection{Further remarks}

In  \cite[Theorem~1.8]{WatanabeYoshidaTwo}
Watanabe and Yoshida proved the following generalization of \cite[Theorem~4.3]{Ooishi1}.

\begin{theorem}\label{WY}
Let $(R, \mf m)$ be a Cohen--Macaulay local ring of dimension $d$  and $I$ be an $\mf m$-primary ideal. 
Then 
\[
\ehk(I^n) \leq \eh(I) \binom{n + d - 2}{d} + \ehk(I)\binom{n + d -2}{d - 1}
= \eh(I) \binom{n + d - 1}{d} - (\eh(I) - \ehk(I)) \binom{n + d - 2}{d-1}
\]
Moreover, if $R$ is weakly F-regular and analytically unramified\footnote{It was brought to my attention by Kriti Goel that these assumptions are missing in the published version of this article.} then 
the equality holds for some $n \geq 2$ if and only if $I$ is stable, i.e., $I^2 = IJ$ for a minimal reduction $J$ of $I$.
\end{theorem}

This theorem provides us an evidence for Question~\ref{q decomposition} and a recipe 
for computing $\ehk(I^n)$. For example, it can be applied to integrally closed ideals in a two-dimensional 
rational singularity. 

\begin{corollary}[Northcott-type inequality]
Let $(R, \mf m)$ be a Cohen--Macaulay local ring of dimension $d$ and $I$ be an $\mf m$-primary ideal. 
Then 
\[
\lim_{q \to \infty} \frac{\eh_1(\frq{I})}{q^d} \geq \eh(I) - \ehk(I).  
\]
\end{corollary}
\begin{proof}
We may use Theorem~\ref{WY} and Proposition~\ref{decompose}, or directly apply Northcott's inequality.
Namely, in \cite[Theorem~1]{Northcott} Northcott proved that every $\mf m$-primary ideal $J$
in a Cohen--Macaulay ring satisfies the inequality
$\eh_1(J) \geq \eh(J) - \length (R/J)$.
Applying this to $J = \frq{I}$ and passing to the limit, we see that 
\[
\lim_{q \to \infty} \frac{\eh_1(\frq{I})}{q^d} \geq \eh(I) - \ehk(I). 
\]
\end{proof}

Since the second coefficient of $\ehk(I^n)$ is the limit of the first Hilbert coefficients, naturally it should carry some information 
about $I$. 
Huneke and Ooishi (\cite[Theorem~2.1]{Huneke}, \cite[Theorem~3.3]{Ooishi}) showed that Northcott's inequality is equality 
if and only if $I$ is stable. Comparing this with Theorem~\ref{WY}, we are naturally led to the following speculation.

\begin{conjecture}
Let $(R, \mf m)$ be a weakly F-regular analytically unramified\footnote{As in Theorem~\ref{WY} these assumptions are missing in the published version. A recent preprint \cite{BGV} gives a counter-example to the published conjecture. This example is not weakly F-regular.} Cohen-Macaulay local ring.
Then an $\mf m$-primary ideal  $I$ is stable if and only if 
$\lim\limits_{q \to \infty} \frac{\eh_1(\frq{I})}{q^d} = \eh(I) - \ehk(I)$.
\end{conjecture}

The conjecture will follow if we will be able to show that
\[
\ehk(I^n) = \eh(I) \binom{n + d -1}{d} - \lim_{q \to \infty} \frac{\eh_1(\frq{I})}{q^d} \binom{n + d - 2}{d - 1} +
\lim_{q \to \infty} \frac{\eh_2(\frq{I})}{q^d} \binom{n + d - 3}{d - 2} + o(n^{d - 2}).
\]
Then, since $\eh_2(J) \geq 0$ by \cite{Narita}, it follows that for $n \gg 0$
 \[
\ehk(I^n) \geq \eh(I) \binom{n + d -1}{d} - \lim_{q \to \infty} \frac{\eh_1(\frq{I})}{q^d} \binom{n + d - 2}{d - 1},
\]
so if $\lim\limits_{q \to \infty} \frac{\eh_1(\frq{I})}{q^d} = \eh(I) - \ehk(I)$, then Theorem~\ref{WY}
shows that $I$ must be stable.

We may also give upper bounds on the limit. For example, Elias (\cite[Proposition~2.1]{Elias})
showed that in a Cohen--Macaulay ring of dimension at least one we have
\[
\eh_1(J) \leq (\eh(R) - 1) (\eh(J) - \eh(R) \ord (J)) + \eh_1(R).
\]
Taking $J = \frq{I}$ and passing to the limit we get
\[
\lim_{q \to \infty} \frac{\eh_1(\frq{I})}{q^d} \leq (\eh(R) - 1)\eh(I).
\]
Note that there exists a fixed $C$ such that $\mf m^{Cq} \subseteq \frq{I}$, 
thus $\lim\limits_{q \to \infty } \ord (\frq{I})/q^d = 0$ if $d > 1$.

\section*{Acknowledgements}

The author thanks Alessandro De Stefani, Kei-Ichi Watanabe, and Ken-Ichi Yoshida for useful discussions. 

\bibliographystyle{alpha}
\bibliography{powers}

\end{document}